\documentclass{amsproc}
\usepackage{fullpage}
\usepackage{amssymb}
\usepackage{amsmath}
\usepackage{amsthm}
\usepackage{amsfonts}
\usepackage{dsfont}
\usepackage{mathrsfs}
\usepackage{graphics}
\usepackage{graphicx}
\usepackage{url}
\usepackage{hyperref}
\usepackage{microtype}
\usepackage{enumitem}
\usepackage{mathrsfs}
\DeclareSymbolFontAlphabet{\mathrsfs}{rsfs}
\usepackage[mathscr]{euscript}
\usepackage{epsfig}

\newcommand{\cone}{\mathbb{C}}
\newcommand{\cstar}{\mathbb{C}^{\times}}
\newcommand{\cpone}{\mathbb{P}^1}

\newcommand{\zed}{\mathbb{Z}}

\newcommand{\xstar}{\mathrsfs{X}_k^*}

\newcommand{\ccwarrow}{\text{\Large$\curvearrowleft$}}

\newtheorem{theorem}{Theorem}[section]
\newtheorem{lemma}[theorem]{Lemma}

\theoremstyle{definition}

\newtheorem{example}[theorem]{Example}

\theoremstyle{remark}

\numberwithin{equation}{section}

\newtheoremstyle{citing}{}{}{\itshape}{}{\bfseries}{.}{ }{\thmnote{#3}}
\theoremstyle{citing}

\begin{document}
\title{A construction of hyperk\"{a}hler metrics through Riemann-Hilbert problems II}


%    author one information
\author{C. Garza}
\address{Department of Mathematics, IUPUI, Indianapolis, USA}
\curraddr{}
\email{cegarza@iu.edu}
\thanks{}

%    author two information
%\author{}
%\address{}
%\curraddr{}
%\email{}
%\thanks{}

\subjclass[2010]{Primary }
%    The 2010 edition of the Mathematics Subject Classification is
%    now available.  If you are citing a classification from the
%    new scheme, use the following input coding instead.
%\subjclass[2010]{Primary }

\date{}

\begin{abstract}
We develop the theory of Riemann-Hilbert problems necessary for the results in \cite{theory}. In particular, we obtain solutions for a family of non-linear Riemann-Hilbert problems through classical contraction principles and saddle-point estimates. We use compactness arguments to obtain the required smoothness property on solutions. We also consider limit cases of these Riemann-Hilbert problems where the jump function develops discontinuities of the first kind together with zeroes of a specific order at isolated points in the contour. Solutions through Cauchy integrals are still possible and they have at worst a branch singularity at points where the jump function is discontinuous and a zero for points where the jump vanishes.
\end{abstract}

\maketitle

\section{Introduction}\label{intr}

This article presents the analytic results needed in \cite{theory}. As stated in said article, in order to construct complete hyperk\"{a}hler metrics $g$ on a special case of complex integrable systems (where moduli spaces of Higgs bundles constitute the prime example of this), one must obtain solutions to a particular infinite-dimensional, nonlinear family of Riemann-Hilbert problems. The analytical methods used to obtain these solutions and the smoothness results can be studied separately from the geometric motivations and so we present them in this article.

For limiting values of the parameter space, the Riemann-Hilbert problem degenerates in the sense that discontinuities appear in the jump function $G(\zeta)$ at $\zeta = 0$ and $\zeta = \infty$ in the contour $\Gamma$. Moreover, $G(\zeta)$ may vanish at isolated pairs of points in $\Gamma$. We study the behavior of the solutions to this boundary value problem near such singularities and we obtain their general form, proving that these functions do not develop singularities even in the presence of these pathologies, thus proving the existence of the hyperk\"{a}hler metrics in \cite{theory}.

The paper is organized as follows:

In Section \ref{statement} we state the Riemann-Hilbert problems to be considered. As shown in \cite{theory}, this arises from certain complex integrable systems satisfying a set of axioms motivated by the theory of moduli spaces of Higgs bundles, but we shall not be concerned about the geometric aspects in this paper.

In Section \ref{solutions} we solve the Riemann-Hilbert problem by iterations running estimates based on saddle-point analysis. Under the right Banach space, these estimates show that we have a contraction, proving that solutions exist and are unique. We then apply the Arzela-Ascoli theorem and uniform estimates to show that the solutions are smooth with respect to the parameter space. 

In Section \ref{special} we consider the special case when the parameter $a$ approaches 0 yielding a Riemann-Hilbert problem whose jump function has discontinuities and zeroes along the contour. We apply Cauchy integral techniques to obtain the behavior of the solutions near the points on the contour with these singularities. We show that a discontinuity of the jump function induces a factor $\zeta^\eta$ in the solutions, where $\eta$ is determined by the discontinuities of the jump function $G$. A zero of order $k$ at $\zeta_0$ induces a factor $(\zeta - \zeta_0)^k$ on the left-side part of the solutions. The nature of these solutions are exploited in \cite{theory} to reconstruct a holomorphic symplectic form $\varpi(\zeta)$ and, ultimately, a hyperk\"{a}hler metric.

\textbf{Acknowledgment:} The author would like to thank Professor Alexander Its for many illuminating conversations that greatly improved the manuscript.

\section{Formulation of the Riemann-Hilbert problem}\label{statement}

\subsection{Monodromy data}

We state the monodromy data we will use in this paper. For a more geometric description of the assumptions we make, see \cite{theory}.

Since we will only consider the manifolds $\mathcal{M}$ in \cite{theory} from a local point of view, we can consider it as a trivial torus fibration. With that in mind, here are the key ingredients we need in order to define our Riemann-Hilbert problem:

\begin{enumerate}[label=\textnormal{(\arabic*)}]
\item A neighborhood $U$ of $0$ in $\cone$ with coordinate $a$. On $U$ we have a trivial torus fibration $U \times T^2 := U \times (S^1)^2$ with $\theta_1, \theta_2$ the torus coordinates.

\item $\Gamma \cong \zed^2$ is a lattice equipped with an antisymmetric integer valued pairing $\left\langle , \right\rangle$. We also assume we can choose primitive elements $\gamma_1, \gamma_2$ in $\Gamma$ forming a basis for the lattice and such that $\left\langle \gamma_1, \gamma_2\right\rangle = 1$.

\item A homomorphism $Z$ from $\Gamma$ to the space of holomorphic functions on $U$.

\item A function $\Omega : \Gamma \to \zed$ such that $\Omega(\gamma) = \Omega(-\gamma), \gamma \in \Gamma$ and such that, for some $K > 0$,
\begin{equation}\label{support}
\frac{|Z_\gamma|}{\left\| \gamma \right\|} > K
\end{equation}
for a positive definite norm on $\Gamma$ and for all $\gamma$ for which $\Omega(\gamma) \neq 0$.\label{omega}
\end{enumerate}

For the first part of this paper we work with the extra assumption
\begin{enumerate}
\item[(5)] $Z_{\gamma_1}(a), Z_{\gamma_2}(a) \neq 0$ for any $a$ in $U$.
\end{enumerate}
Later in this paper we will relax this condition.

Observe that the torus coordinates $\theta_1, \theta_2$ induce a homomorphism $\theta$ from $\Gamma$ to the space of functions on $T^2$ if we assign $\gamma_{k} \mapsto \theta_{k}, k = 1, 2$. We denote by $\theta_\gamma, \gamma \in \Gamma$ the result of this map.
  
We consider a different complex plane $\cone$ with coordinate $\zeta$. Let $R > 0$ be an extra real parameter that we consider. We define the ``semiflat'' functions $\mathcal{X}_{\gamma} : U \times T^2 \times \cstar \to \cstar$ for any $\gamma \in \Gamma$ as
\begin{equation}\label{semiflat}
\mathcal{X}^\text{sf}_\gamma (a, \theta_1, \theta_2, \zeta) = \exp\left( \pi R \frac{Z_{\gamma}(a)}{\zeta} + i\theta_\gamma + \pi R \zeta \overline{Z_\gamma(a)}\right) 
\end{equation}
As in the case of the map $\theta$, it suffices to define $\mathcal{X}^\text{sf}_{\gamma_1}$ and $\mathcal{X}^\text{sf}_{\gamma_2}$.

For each $a \in U$ and $\gamma \in \Gamma$ such that $\Omega(\gamma) \neq 0$, the function $Z_\gamma$ defines a ray $\ell_\gamma(a)$ in $\cone$ given by
\[ \ell_\gamma(a) = \{ \zeta \in \cone : \zeta = -t Z_\gamma(a), t > 0 \} \]

Given a pair of functions $\mathcal{X}_k :  U \times T^2 \times \cstar \to \cone, k = 1, 2$, we can extend this with the basis $\{\gamma_1, \gamma_2\}$ as before to a collection of functions $\mathcal{X}_\gamma, \gamma \in \Gamma$. Each element $\gamma$ in the lattice also defines a transformation $\mathcal{K}_\gamma$ for these functions in the form
\[ \mathcal{K}_\gamma \mathcal{X}_{\gamma'} = \mathcal{X}_{\gamma'} (1 - \mathcal{X}_{\gamma})^{\left\langle \gamma', \gamma \right\rangle} \]

For each ray $\ell$ from 0 to $\infty$ in $\cone$ we can define a transformation
\begin{equation}\label{stkfac}
S_\ell = \prod_{\gamma : \ell_\gamma(u) = \ell} \mathcal{K}_\gamma^{\Omega(\gamma)}
\end{equation}
Observe that all the $\gamma$'s involved in this product are multiples of each other, so the $\mathcal{K}_\gamma$ commute and the order for the product is irrelevant.

We can now state the main type of Riemann-Hilbert problem we consider in this paper. We seek to obtain two functions $\mathcal{X}_k :  U \times T^2 \times \cstar \to \cstar, k = 1, 2$ with the following properties:
\begin{enumerate}
	\item Each $\mathcal{X}_k$ depends piecewise holomorphically on $\zeta$, with discontinuities only at the rays $\ell_\gamma(a)$ for which $\Omega(\gamma) \neq 0$. The functions are smooth on $U \times T^2$.
	\item The limits $\mathcal{X}_k^{\pm}$ as $\zeta$ approaches any ray $\ell$ from both sides exist and are related by
	\begin{equation}\label{invjmp}
	\mathcal{X}^+_k = S_\ell^{-1} \circ \mathcal{X}^-_k
	\end{equation}
	\item $\mathcal{X}$ obeys the reality condition
	\[ \overline{\mathcal{X}_{-\gamma}(-1/\overline{\zeta})} = \mathcal{X}_\gamma(\zeta) \]
	\item For any $\gamma \in \Gamma$, $\lim_{\zeta \to 0} \mathcal{X}_\gamma(\zeta) / \mathcal{X}^{\text{sf}}_\gamma(\zeta)$ exists and is real.
\end{enumerate}

\subsection{Isomonodromic Deformation}

It will be convenient for the geometric applications to move the rays to a contour that is independent of $a$. Even though the rays $\ell_\gamma$ defining the contour for the Riemann-Hilbert problem above depend on the parameter $a$, we can assume the open set $U \subset \cone$ is small enough so that there is a pair of rays $r, -r$ such that for all $a \in U$, half of the rays  lie inside the half-plane $\mathbb{H}_r$ of vectors making an acute angle with $r$; and the other half of the rays lie in $\mathbb{H}_{-r}$. We call such rays \textit{admissible rays}. We are allowing the case that $r$ is one of the rays $\ell_\gamma$, as long as it satisfies the above condition.
\begin{figure}[htbp]
	\centering
		\includegraphics[width=0.40\textwidth]{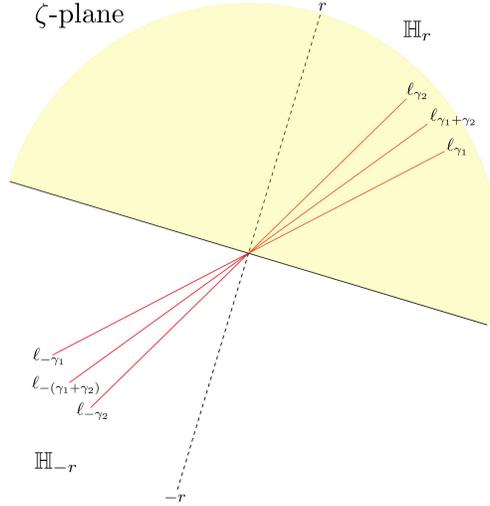}
	\caption{Admissible rays as the new contours for Riemann-Hilbert problems}
	\label{3rays}
\end{figure}
For $\gamma \in \Gamma$, we define $\gamma > 0$ (resp. $\gamma < 0$) as $\ell_\gamma \in \mathbb{H}_r$ (resp. $\ell_\gamma \in \mathbb{H}_{-r}$). Our Riemann-Hilbert problem will have only two anti-Stokes rays, namely $r$ and $-r$.  In this case, the Stokes factors are the concatenation of all the Stokes factors $S^{-1}_\ell$ in \eqref{stkfac} in the counterclockwise direction:
\begin{align*}
S_r & = \prod^\ccwarrow_{\gamma > 0}{\mathcal{K}^{\Omega(\gamma; a)}_\gamma}\\
S_{-r} & = \prod^\ccwarrow_{\gamma < 0}{\mathcal{K}^{\Omega(\gamma; a)}_\gamma}
\end{align*}
Thus, we reformulate the Riemann-Hilbert problem in terms of two functions $\mathcal{Y}_k :  U \times T^2 \times \cstar \to \cstar, k = 1, 2$ with discontinuities at the admissible rays $r, -r$ by replacing condition \ref{invjmp} above with
\begin{equation}\label{newjumps} \begin{array}{rll}
	\mathcal{Y}^+_k & = S_r \circ \mathcal{Y}^-_k, & \qquad \text{along $r$} \\
	\mathcal{Y}^+_k & = S_{-r} \circ \mathcal{Y}^-_k, & \qquad \text{along $-r$}
	\end{array}
\end{equation}
The other conditions remain the same:
\begin{enumerate}
	\item  The functions $\mathcal{Y}_k$ are smooth on $U \times T^2$.\label{smoothprop}
	\item $\mathcal{Y}$ obeys the reality condition
	\[ \overline{\mathcal{Y}_{-\gamma}(-1/\overline{\zeta})} = \mathcal{Y}_\gamma(\zeta) \]\label{realt}
	\item For any $\gamma \in \Gamma$, $\lim_{\zeta \to 0} \mathcal{Y}_\gamma(\zeta) / \mathcal{X}^{\text{sf}}_\gamma(\zeta)$ exists and is real. \label{asympt}
\end{enumerate}

In the following section we will prove the main theorem of this paper:

\begin{theorem}\label{solut}
There exists a pair of functions $\mathcal{Y}_k :  U \times T^2 \times \cstar \to \cstar, k = 1, 2$ satisfying \eqref{newjumps} and conditions \eqref{smoothprop}, \eqref{realt}, \eqref{asympt}. These functions are unique up to multiplication by a real constant.
\end{theorem}

\section{Solutions}\label{solutions}

We start working on a proof of Theorem \ref{solut}. As in the classical scalar Riemann-Hilbert problems, we obtain the solutions $\mathcal{Y}_k$ by solving the integral equation
\begin{equation}\label{logspm}
\mathcal{Y}_k(a,\zeta) = \mathcal{X}_{\gamma_k}^{\text{sf}}(a,\zeta) \exp \left( \frac{1}{4\pi i} \left\{ \int_r K(\zeta, \zeta') \log(S_r \mathcal{Y}_k) + \int_{-r} K(\zeta, \zeta') \log(S_{-r} \mathcal{Y}_k) \right\} \right), \quad k = 1, 2
\end{equation}
where we abbreviated  $\dfrac{d\zeta'}{\zeta'}\dfrac{\zeta'+\zeta}{\zeta'-\zeta}$ as $K(\zeta',\zeta)$. The dependence of $\mathcal{Y}_k$ on the torus coordinates $\theta_1, \theta_2$ has been omitted to simplify notation. We will write $\mathcal{Y}_\gamma$ to denote the function resulting from the (multiplicative) homomorphism from $\Gamma$ to nonzero functions on $U \times T^2 \times \cstar$ induced by $\mathcal{Y}_k, k = 1, 2$. 

It will be convenient to write
\begin{equation}\label{thet}
 \mathcal{Y}_\gamma(a, \zeta, \theta) = \mathcal{X}_\gamma^{\text{sf}}(a, \zeta, \Theta),
\end{equation}
for $\Theta_k : U \times  T^2 \times \cstar \to \cone, k = 1, 2$. We abuse notation and write $\theta$ for $(\theta_1, \theta_2)$, as we do with $\Theta$.

If we take the power series expansion of $\log(S_r \mathcal{Y}_k), \log(S_{-r} \mathcal{Y}_k)$ and decompose the terms into their respective components in each $\gamma \in \Gamma$, we can rewrite the integral equation \eqref{logspm} as
\begin{equation}\label{intsto}
\mathcal{Y}_\gamma(a,\zeta) = \mathcal{X}_\gamma^{\text{sf}}(a,\zeta)\exp \left( \frac{1}{4\pi i}\left\{  \sum_{\gamma' > 0} f^{\gamma'} \int_r K(\zeta,\zeta') \mathcal{Y}_{\gamma'}(a,\zeta') +  \sum_{\gamma' < 0} f^{\gamma'} \int_{-r} K(\zeta,\zeta') \mathcal{Y}_{\gamma'}(a,\zeta')\right\}\right)
\end{equation}
where 
\[ f^{\gamma'} = c_{\gamma'} \left\langle \gamma, \gamma' \right\rangle, \]
$c_{\gamma'}$ a rational constant obtained by power series expansion. 

\begin{example}[The Pentagon case]

As our main example of this families of Riemann-Hilbert problems, we have the Pentagon case, studied in more detail in \cite{theory}. Here the jump functions $S_r, S_{-r}$ are of the form
\begin{align}
\left. \begin{array}{ll}
          \mathcal{Y}_1 & \mapsto \mathcal{Y}_1(1-\mathcal{Y}_2)\\
          \mathcal{Y}_2 & \mapsto \mathcal{Y}_2(1-\mathcal{Y}_1(1-\mathcal{Y}_2))^{-1}
          \end{array}  \right\} & S_r \label{newj1}\\
\intertext{and, similarly}
\left. \begin{array}{ll}
          \mathcal{Y}_1 & \mapsto \mathcal{Y}_1(1-\mathcal{Y}^{-1}_2)^{-1}\\
          \mathcal{Y}_2 & \mapsto \mathcal{Y}_2(1-\mathcal{Y}^{-1}_1(1-\mathcal{Y}^{-1}_2))          
          \end{array} \right\} & S_{-r} \label{newj2}
\end{align}

If we expand $\log(S_{r}  \mathcal{Y}_k), k = 1, 2$ etc. we obtain
\begin{equation*}
   f^{i\gamma_{1}+j\gamma_2}= \left\{ \begin{array}{ll}
                      \dfrac{-1}{j} \left\langle \gamma, \gamma_{2} \right\rangle & \text{ if $i=0$}\\
                       \dfrac{ (-1)^{j}  }{i} \binom{|i|}{|j|} \left\langle \gamma, \gamma_{1} \right\rangle & \text{ if $0\leq j\leq i$ or $i \leq j \leq 0$}\\
                      0 & \text{ otherwise.}
                      \end{array} \right.
\end{equation*}
\end{example}

Back in the general case, our approach for a solution to \eqref{intsto} is to work with iterations. For $\nu \in \mathbb{N}$:
\begin{equation}\label{itery}
\mathcal{Y}^{(\nu+1)}_\gamma(a,\zeta) = \mathcal{X}_\gamma^{\text{sf}}(a,\zeta)\exp \left( \frac{1}{4\pi i}\left\{  \sum_{\gamma' > 0} f^{\gamma'} \int_r K(\zeta,\zeta') \mathcal{Y}^{(\nu)}_{\gamma'}(a,\zeta') +  \sum_{\gamma' < 0} f^{\gamma'} \int_{-r} K(\zeta,\zeta') \mathcal{Y}^{(\nu)}_{\gamma'}(a,\zeta')\right\}\right)
\end{equation}
Formula \eqref{itery} requires an explanation. Assuming $\mathcal{Y}^{(\nu-1)}_{\gamma'}, \gamma' \in \Gamma$ has been constructed, by definition, $\mathcal{Y}^{(\nu)}_{\gamma'}$ has jumps at $r$ and $-r$. By abuse of notation, $\mathcal{Y}^{(\nu)}_{\gamma'}$ in \eqref{itery} denotes the analytic continuation to the ray $r$ (resp. $-r$) along $\mathbb{H}_r$ (resp. $\mathbb{H}_{-r}$) in the case of the first (resp. second) integral.

By using \eqref{thet}, we can write \eqref{intsto} as an additive Riemann-Hilbert problem where we solve the integral equation
\begin{equation}\label{ineqs}
e^{i\Theta_\gamma} = e^{i\theta_\gamma}\exp \left( \frac{1}{4\pi i}\left\{ \sum_{\gamma' > 0} f^{\gamma'} \int_r K(\zeta,\zeta') \mathcal{X}^{\text{sf}}_{\gamma'}(a,\zeta',\Theta) + \sum_{\gamma' < 0} f^{\gamma'} \int_{-r} K(\zeta,\zeta') \mathcal{X}^{\text{sf}}_{\gamma'}(a,\zeta',\Theta)\right\}\right)
\end{equation}

As in \eqref{itery}, the solution of \eqref{ineqs} is obtained through iterations:
\begin{equation}\label{theta0}
\Theta^{(0)}(\zeta,\theta) = \theta,
\end{equation}
\begin{align}
e^{i\Theta_\gamma^{(\nu+1)}} = e^{i\theta_\gamma}\exp \left(\frac{1}{4\pi i}\left\{ \sum_{\gamma' > 0} f^{\gamma'} \int_r K(\zeta,\zeta') \mathcal{X}^{\text{sf}}_{\gamma'}(a,\zeta',\Theta^{(\nu)}) + \sum_{\gamma' < 0} f^{\gamma'} \int_{-r} K(\zeta,\zeta') \mathcal{X}^{\text{sf}}_{\gamma'}(a,\zeta',\Theta^{(\nu)})\right\}\right) \label{upsto}
\end{align}

We need to show that $\Theta^{(\nu)} = (\Theta_1^{(\nu)}, \Theta_2^{(\nu)})$ converges uniformly in $a$ to well defined functions $\Theta_k : U \times T^2 \times \cpone \to \cone, k = 1, 2$ with the right smooth properties on $a$ and $\zeta$. Define $\mathrsfs{X}$ as the completion of the space of bounded functions of the form $\Phi: U \times T^2 \times \cpone \to \cone^{2}$ that are smooth on $U \times T^2$ under the norm
\begin{equation}\label{normd}
\left\| \Phi \right\| =  \sup_{\zeta,\theta, a} \left\| \Phi(\zeta, \theta, a) \right\|_{\cone^{2}},
\end{equation}
where $\cone^{2}$ is assumed to have as norm the maximum of the Euclidean norm of its coordinates. Notice that we have not put any restriction on the functions $\Phi$ in the $\cpone$ slice, except that they must be bounded. Our strategy will be to solve the Riemann-Hilbert problem in  $\mathrsfs{X}$ and show that for sufficiently big (but finite) $R$, we can get uniform estimates on the iterations yielding such solutions and any derivative with respect to the parameters $a, \theta$. The Arzela-Ascoli theorem will give us that the solution $\Phi$ not only lies in $\mathrsfs{X}$, but it preserves all the smooth properties. The very nature of the integral equation will guarantee that its solution is piecewise holomorphic on $\zeta$, as desired.

We're assuming as in \cite{gaiotto} that $\Gamma$ has a positive definite norm satisfying the Cauchy-Schwarz property
\[ \left|\left\langle \gamma, \gamma' \right\rangle \right| \leq \left\|\gamma\right\| \left\|\gamma' \right\| \]
as well as the ``Support property'' \eqref{support}.
%\begin{equation}\label{supprt}
% \left\|\gamma\right\| < \text{const}|Z_\gamma|,
%\end{equation}
%for all $\gamma$ such that $\Omega(\gamma) \neq 0$.
For any $\Phi \in \mathrsfs{X}$, let $\Phi_k$ denote the composition of $\Phi$ with the $k$th projection $\pi_k : \cone^{2} \to \cone, k = 1, 2$. Instead of working with the full Banach space $\mathrsfs{X}$, let $\mathrsfs{X}^*$ be the collection of $\Phi \in \mathrsfs{X}$ in the closed ball
\begin{equation}\label{bounds}
\left\| \Phi - \theta \right\| \leq \epsilon,
\end{equation}
for an $\epsilon > 0$ so small that
\begin{equation}\label{condex}
\sup_{\zeta,\theta,a} \left|  e^{i\Phi_k}  \right| \leq 2, 
\end{equation}
for $k =  1, 2$. In particular, $\mathrsfs{X}^*$ is closed, hence complete. Note that by \eqref{condex}, if $\Phi \in \mathrsfs{X}^*$, then $e^{i\Phi} \in \mathrsfs{X}$. Furthermore, by \eqref{upsto}, the transformation in $\zeta$ is only as an integral transformation, so $\Theta^{(\nu)}$ is holomorphic in either of the half planes $\mathbb{H}_r$ or $\mathbb{H}_{-r}$. 

\subsection{Saddle-point Estimates}

We will prove the first of our uniform estimates on $\Theta^{(\nu)}$.

\begin{lemma}\label{nubst}
$\Theta^{(\nu)} \in \xstar$ for all $\nu$.
\end{lemma}
\begin{proof}
We follow \cite{gaiotto}, using induction on  $\nu$. The statement is clearly true for $\nu = 0$ by \eqref{theta0}. Assuming $\Theta^{(\nu)} \in \mathrsfs{X}^*$, take the $\log$ in both sides of \eqref{upsto}:
\begin{equation}\label{logthet}
 \Theta^{(\nu+1)}_k - \theta_k = -\frac{1}{4\pi }\left\{ \sum_{\gamma' > 0} f^{\gamma'} \int_r K(\zeta,\zeta') \mathcal{X}^{\text{sf}}_{\gamma'}(a,\zeta',\Theta^{(\nu)}) + \sum_{\gamma' < 0} f^{\gamma'} \int_{-r} K(\zeta,\zeta') \mathcal{X}^{\text{sf}}_{\gamma'}(a,\zeta',\Theta^{(\nu)}) \right\}, \quad k = 1, 2
  \end{equation}
For general $\Phi \in \mathrsfs{X}^*$, $\Phi$ can be very badly behaved in the $\cpone$ slice, but by our inductive construction, $\Theta^{(\nu+1)}$ is even holomorphic in $\mathbb{H}_r$ and $\mathbb{H}_{-r}$. Consider the integral
\begin{equation}\label{integr}
\int_r K(\zeta,\zeta') \mathcal{X}^{\text{sf}}_{\gamma'}(a,\zeta',\Theta^{(\nu)})
\end{equation}
The function $\Theta^{(\nu)}$ can be analytically extended along the ray $r$ so that it is holomorphic on the sector $V$ bounded by $r$ and $\ell_{\gamma'}, \gamma' > 0$ (see Figure \ref{vsect}). By Cauchy's theorem, we can move \eqref{integr} to one along the ray $\ell_{\gamma'}$, possibly at the expense of a residue of the form

\begin{figure}[htbp]
	\centering
		\includegraphics[width=0.40\textwidth]{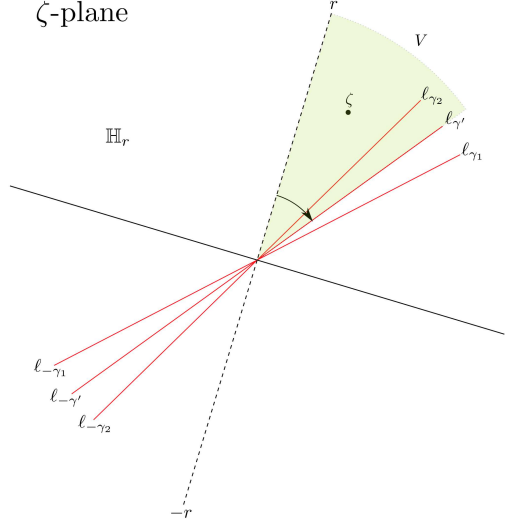}
	\caption{Translating the integral to the ray $\ell_{\gamma'}$}
	\label{vsect}
\end{figure}

\begin{equation}\label{resid}
4\pi i \exp\left[i\Theta_{\gamma'}^{(\nu)} + \pi R \left( \frac{Z_{\gamma'}}{\zeta} + \overline{Z_{\gamma'}}\zeta \right) \right]
\end{equation}
if $\zeta$ lies in $V$. This residue is in control. Indeed, by the induction hypothesis, $\left| e^{i \Theta_{\gamma'}^{(\nu)}} \right| < 2^{\left\| \gamma'\right\|}$, independent of $\nu$. Moreover, we pick a residue only if $\zeta$ lies in the sector $S$ bounded by the first and last $\ell_{\gamma_k}, \gamma_k \in \{\gamma_1, \gamma_{2}\}$ included in $\mathbb{H}_r$ traveling in the counterclockwise direction. This sector is strictly smaller than $\mathbb{H}_r$ (see Figure \ref{zetar}), so $\arg Z_{\gamma'} - \arg \zeta \in (-\pi,\pi)$ and, since $r$ makes an acute angle with all rays $\ell_{\gamma'}, \gamma'>0$:

\begin{figure}[htbp]
	\centering
		\includegraphics[width=0.40\textwidth]{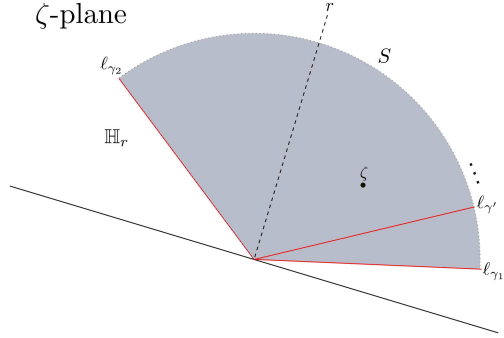}
	\caption{A residue appears only if $\zeta$ lies in $S$}
	\label{zetar}
\end{figure}

\[  |\arg Z_{\gamma'} - \arg \zeta| > \text{const} > \frac{\pi}{2} 
    \hspace{5 mm} \text{ for all $\gamma' > 0, \zeta \in S$.} \]
In particular,
\begin{equation}\label{cosine}
 \cos(\arg Z_{\gamma'} - \arg \zeta) < -\text{const} <0 \hspace{5 mm} \text{ for all $\gamma' > 0, \zeta \in S$.}
 \end{equation}
Using the fact that $\inf (|\zeta| + 1/|\zeta|) = 2$, the sum of residues of the form \eqref{resid} is bounded by:
\begin{equation}\label{resest}
 \sum_{\gamma' >0} \left| f^{\gamma'} \right| 2^{\left\|  
\gamma' \right\|} e^{-\text{const} R|Z_{\gamma'}|} \,
\end{equation}
Recall that $\left\| \gamma'\right\| < \text{const} |Z_{\gamma'}|$, so \eqref{resest} can be simplified to
\begin{equation}\label{bstest}
 \sum_{\gamma' >0} \left|f^{\gamma'} \right|  e^{(-\text{const} R + \delta)|Z_{\gamma'}|} \,
\end{equation}
for a constant $\delta$. We're assuming that $\Omega(\gamma')$ do not grow too quickly with $\gamma'$, by the support property \eqref{support}, so $\left|f^{\gamma'} \right|$ is dominated by the exponential term and the above sum can be made arbitrarily small if $R$ is big enough. This bound can be chosen to be independent of $\nu$, $\zeta$ and the basis element $\gamma_k$ (by choosing the maximum among the $\gamma_1, \gamma_{2}$). The exact same argument can be used to show that the residues of the integrals along $-r$ are in control. In fact, let $\epsilon > 0$ be given. Choose $R > 0$ so that the total sum of residues Res$(\zeta)$ is less than $\epsilon/2$.

Thus, we can assume the integrals are along $\ell_{\gamma'}$ and consider
\begin{equation*}%\label{intsadl}
\int_{\ell_{\gamma'}} K(\zeta,\zeta') \mathcal{X}^{\text{sf}}_{\gamma'}(a,\zeta',\Theta^{(\nu)})
\end{equation*}
The next step is to do a saddle point analysis and obtain the asymptotics for large $R$. Since this type of analysis will be of independent interest to us, we leave these results to a separate Lemma at the end of this section.

By \eqref{condex}, $\left| \exp \left( i \Theta^{(\nu)}_{\gamma'}(\zeta_0)\right) \right| \leq 2^{\left\| \gamma'\right\|}$.  Thus, by Lemma \ref{saddleresults}, for $\zeta$ away from the saddle $\zeta_0$, we can bound the contribution from the integral by
\begin{equation}\label{festm}
\text{const } \left| f^{\gamma'}\right| 2^{\left\| \gamma'\right\|} \frac{e^{-2\pi R|Z_{\gamma'}|}}{\sqrt{R |Z_{\gamma'}|}}
\end{equation}
if $R$ is big enough.

The case of $\zeta = \zeta_0$ is, by Lemma \ref{saddleresults}, as in \eqref{festm} except without the $\sqrt{R}$ term in the denominator. In any case, by \eqref{sdrest}, and since $\exp \left( i \Theta^{(\nu)}_{\gamma'}(\zeta_0)\right) \leq 2^{\left\|\gamma' \right\|}$ by \eqref{condex} and by \eqref{festm},

\begin{equation}\label{spest}
\left| \sum_{\gamma'} f^{\gamma'}\int_{\ell_{\gamma'}} K(\zeta,\zeta') \mathcal{X}^{\text{sf}}_{\gamma'}(a,\zeta',\Theta^{(\nu)}) \right| <
\text{const } \sum_{\gamma'}\left|
 f^{\gamma'} \right| e^{(-2\pi R + \delta)|Z_{\gamma'}|}.
\end{equation}
The $\delta$ constant is the same appearing in \eqref{bstest}. This sum is convergent by the tameness condition on the $\Omega(\gamma')$ coefficients, and can be made arbitrarily small if $R$ is big enough. Putting everything together:
\begin{align*}
\sup_{\zeta,\theta} \left| \Theta^{(\nu+1)}_\gamma - \theta_\gamma \right| & =\text{const } \sum_{\gamma'}\left|
 f^{\gamma'}  \right| e^{(-2\pi R + \delta)|Z_{\gamma'}|} + \text{Res$(\zeta)$}\\
& < \frac{\epsilon}{2} + \frac{\epsilon}{2} = \epsilon.
\end{align*}
Therefore $\left\| \Theta^{(\nu+1)} - \theta \right\| < \epsilon$. In particular, $\left\| \Theta^{(\nu+1)} \right\| < \infty$, so $\Theta^{(\nu+1)} \in \mathrsfs{X}^*$. Since $\epsilon$ was arbitrary, $\Theta^{(\nu+1)}$ satisfies the side condition \eqref{condex} and thus $\Theta^{(\nu)} \in \mathrsfs{X}^*$ for all $\nu$ if $R$ is big enough.
\end{proof}

We finish this subsection with the proof of some saddle-point analysis results used in the previous lemma.

\begin{lemma}\label{saddleresults}
For every $\nu$ consider an integral of the form
\begin{equation}\label{intsadl}
F(\zeta) = \int_{\ell_{\gamma'}} K(\zeta,\zeta') \mathcal{X}^{\text{sf}}_{\gamma'}(a,\zeta',\Theta^{(\nu)})
\end{equation}
Let $\zeta_0 = -e^{i \arg Z_{\gamma'}}$. Then, for $\zeta \neq \zeta_0$, we can estimate the above integral as
\begin{equation}\label{sdrest}
F(\zeta) = -\frac{\zeta_0 + \zeta}{\zeta_0 - \zeta} \exp \left( i \Theta^{(\nu)}(\zeta_0)\right) \frac{1}{\sqrt{R|Z_{\gamma'}|}} e^{-2\pi R |Z_{\gamma'}|} + O\left( \frac{e^{-2\pi R |Z_{\gamma'}|}}{R} \right), \qquad \text{as $R \to \infty$}
\end{equation}
For $\zeta = \zeta_0$,
\begin{equation}\label{zeta0est}
 F(\zeta_0) = O\left( \frac{e^{-2\pi R |Z_{\gamma'}|}}{R} \right), \qquad \text{as $R \to \infty$}
\end{equation}
\end{lemma}
\begin{proof}
Equation \eqref{intsadl} is of the type
\begin{equation}\label{stmeth}
h(R) = \int_{\ell_{\gamma'}} g(\zeta') e^{\pi R f(\zeta')}
\end{equation}
where
\[ g(\zeta') = \frac{\zeta'+\zeta}{\zeta'(\zeta'-\zeta)}, \hspace{5 mm} f(\zeta') = \frac{Z_{\gamma'}}{\zeta'} 
+ \zeta' \overline{Z_{\gamma'}}. \]
The function $f$ has a saddle point $\zeta_0 = -e^{i \arg Z_{\gamma'}}$ at the intersection of the ray $\ell_{\gamma'}$ with the unit circle. Moreover, $f(\zeta_0) = -2|Z_{\gamma'}|$. The ray $\ell_{\gamma'}$ and the unit circle are the locus of $\text{Im } f(\zeta') = \text{Im }f(\zeta_0) = 0$. It's easy to see that in $\ell_{\gamma'}$ $f(\zeta') < f(\zeta_0)$ if $\zeta' \neq \zeta_0$, so $\ell_{\gamma'}$ is the path of steepest descent (see Figure \ref{spana}).

\begin{figure}[htbp]
	\centering
		\includegraphics[width=0.40\textwidth]{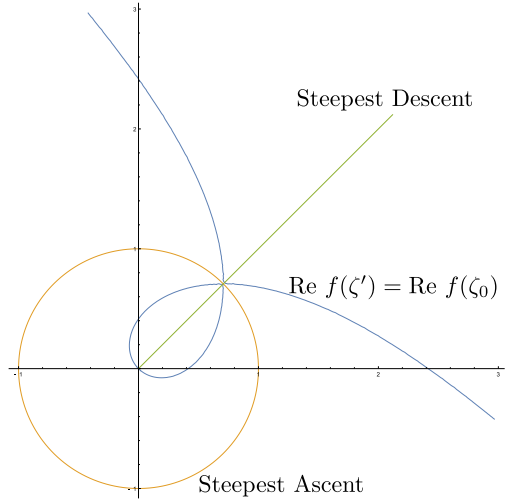}
	\caption{Paths of steepest descent and ascent}
	\label{spana}
\end{figure}

Introduce $\tau$ by

\begin{equation*}
\frac{1}{2} (\zeta' - \zeta_0)^2 f''(\zeta_0) + O((\zeta' -\zeta_0)^3) = -\tau^2
\end{equation*}

and so

\begin{equation}\label{zettau}
\zeta' - \zeta_0 = \left\{ \frac{-2}{f''(\zeta_0)}\right\}^{\frac{1}{2}}\tau + O(\tau^2)
\end{equation}

for an appropriate branch of $\{f''(\zeta_0)\}^{1/2}$. Let $\alpha = \arg f''(\zeta_0) = -2\arg Z_{\gamma'} + \pi$. The branch of $\{f''(\zeta_0)\}^{1/2}$ is chosen so that $\tau > 0$ in the part of the steepest descent path outside the unit disk in Figure \ref{spana}. That is, $\tau > 0$ when $\arg (\zeta' - \zeta_0) = \frac{1}{2}\pi - \frac{1}{2}\alpha$, and so $\{f''(\zeta_0)\}^{1/2} = i\sqrt{2|Z_{\gamma'}|}e^{-i\arg Z_{\gamma'}}$. Thus \eqref{zettau} simplifies to

\begin{equation*}
\zeta' - \zeta_0 = \frac{-\zeta_0}{\sqrt{|Z_{\gamma'}|}}\tau + O(\tau^2)
\end{equation*}

We expand $g(\zeta'(\tau))$ as a power series\footnote{In our case, $g$ depends also on the parameter $R$, so this is an expansion on $\zeta'$}:

\begin{equation}
 g(\zeta'(\tau))  = g(\zeta_0) + g'(\zeta_0) \left\{ \frac{-2}{f''(\zeta_0)} \right\}^{\frac{1}{2}}\tau + O(\tau^2)
\end{equation}

As in \cite{murray}, 

\begin{equation*}
h(R) \sim e^{Rf(\zeta_0)}g(\zeta_0)  \left\{ \frac{-2}{f''(\zeta_0)} \right\}^{\frac{1}{2}} \int_{-\infty}^\infty e^{-R \tau^2} d\tau + \ldots
\end{equation*}
and so
\begin{align*}
h(R) & = \sqrt{\frac{2\pi}{R|f''(\zeta_0)|}} g(\zeta_0) e^{R f(\zeta_0) + (i/2)(\pi - \alpha)} + O\left( \frac{e^{R f(\zeta_0)}}{R} \right) \\
\intertext{in our case, and since $\zeta_0 = -e^{i \arg Z_{\gamma'}}$}
 & = -\frac{\zeta_0 + \zeta}{\zeta_0 - \zeta} \exp \left( i \Theta^{(\nu)}(\zeta_0)\right) \frac{1}{\sqrt{R|Z_{\gamma'}|}} e^{-2\pi R |Z_{\gamma'}|} + O\left( \frac{e^{-2\pi R |Z_{\gamma'}|}}{R} \right), \qquad \text{as $R \to \infty$}
\end{align*}	
This shows \eqref{sdrest}.

%By \eqref{condex}, $\left| \exp \left( i \Theta^{(\nu)}_{\gamma'}(\zeta_0)\right) \right| \leq 2^{\left\| \gamma'\right\|}$.  Thus, for $\zeta$ bounded away from the saddle $\zeta_0$, we can bound the contribution from the integral by
%\begin{equation}\label{festm}
%\text{const } \left| f^{\gamma'}\right| 2^{\left\| \gamma'\right\|} \frac{e^{-2\pi R|Z_{\gamma'}|}}{\sqrt{R |Z_{\gamma'}|}}
%\end{equation}
%if $R$ is big enough. 

If $\zeta \to \zeta_0$, we take a different path of integration, consisting of 3 parts $\ell_1, \ell_2, \ell_3$ (see Figure \ref{newpath}).

\begin{figure}[htbp]
	\centering
		\includegraphics[width=0.30\textwidth]{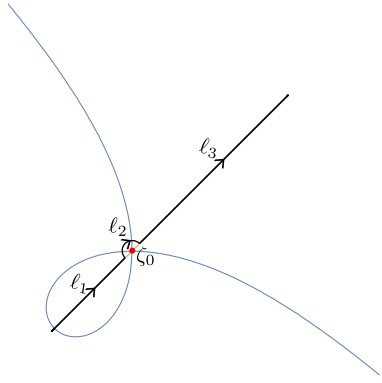}
	\caption{If $\zeta \to \zeta_0$, a modification of the path is required}
	\label{newpath}
\end{figure}

If we parametrize the $\ell_{\gamma'}$ ray as $\zeta' = -e^{t + i\arg Z_{\gamma'}} = - e^t \zeta_0, -\infty < t < \infty$, the $\ell_2$ part is a semicircle around $t = -\epsilon$ and $t = \epsilon$, for small $\epsilon$. The contribution from $\ell_2$ is clearly (up to a factor of $2 \pi i$) half of the residue of the function in \eqref{intsadl}. As in \eqref{resest}, this residue is:
\begin{equation}\label{ressddl}
 2 \pi i \exp \left( i \Theta^{(\nu)}(\zeta_0) -2 \pi R |Z_{\gamma'}|\right).
\end{equation}

If we denote by $\exp \left( i \Theta^{(\nu)}(t) \right)$ the evaluation $\exp \left( i \Theta^{(\nu)}(-t\zeta_0) \right)$, the contributions from $\ell_1$ and $\ell_3$ in the integral are of the form
\begin{align}
\lim_{\epsilon \to 0} \left\{ \vphantom{\int_1^\infty t}\right. & \int_{-\infty}^{-\epsilon} dt \frac{-e^t  + 1}{-e^t  - 1} \exp \left( i \Theta^{(\nu)}(t) \right) \exp \left( \pi R (e^t + e^{-t})\right) \notag\\
& \left. + \int_{\epsilon}^{\infty} dt \frac{-e^t  + 1}{-e^t  - 1} \exp \left( i \Theta^{(\nu)}(t)\right) \exp \left( \pi R (e^t + e^{-t})\right) \right\} \label{2ells}
\end{align}

If we do the change of variables $t \mapsto -t$ in the first integral, \eqref{2ells} simplifies to
\begin{equation}\label{sadend}
\int_0^\infty dt \frac{-e^t  + 1}{-e^t  - 1} \left[ \exp \left( i \Theta^{(\nu)}(t)\right) - \exp \left( i \Theta^{(\nu)}(-t)\right) \right] \exp \left( \pi R (e^t + e^{-t})\right)
\end{equation}

\eqref{sadend} is of the type \eqref{stmeth}, with
\begin{equation*}
g(\zeta') = \frac{\zeta'+\zeta_0}{\zeta'(\zeta'-\zeta_0)} \left[ \exp \left( i \Theta^{(\nu)}(\zeta')\right) - \exp \left( i \Theta^{(\nu)}(1/\zeta')\right) \right]
\end{equation*}

Since $\zeta_0 = 1/\zeta_0$, the apparent pole at $\zeta_0$ of $g(\zeta')$ is removable and the integral can be estimated by the same steepest descent methods as in \eqref{intsadl}. The only difference is that the saddlepoint now lies at one of the endpoints. This only introduces a factor of $1/2$ in the estimates (see \cite{murray}). If $g(\zeta_0) \neq 0$ in this case, the integral is just
\begin{equation}\label{endpest}
\frac{g(\zeta_0)}{2\sqrt{R|Z_{\gamma'}|}} e^{-2\pi R |Z_{\gamma'}|+i \arg Z_{\gamma'}} + O\left( \frac{e^{-2\pi R |Z_{\gamma'}|}}{R} \right)
\end{equation}

If $g(\zeta_0) = 0$, then the estimate is at least of the order $O\left( \frac{e^{-2\pi R |Z_{\gamma'}|}}{R} \right)$. This finishes the proof of \eqref{zeta0est}.
\end{proof}

\subsection{Uniform Estimates on Derivatives}

Now let $\beta = (\beta_1, \beta_{2}, \beta_{3}, \beta_{4})$ be a multi-index in $\mathbb{N}^{4}$, and let $D^\beta$ be a differential operator acting on the iterations $\Theta^{(\nu)}$:
\begin{equation}\label{dbeta}
 D^\beta \Theta^{(\nu)}_\gamma = \frac{\partial}{ \partial \theta_1^{\beta_1} \partial \theta_{2}^{\beta_{2}}\partial a^{\beta_{3}} \partial \overline{a}^{\beta_{4}}} \Theta^{(\nu)}_\gamma
\end{equation}

We need to uniformly bound the partial derivatives of $\Theta^{(\nu)}$ on compact subsets:

\begin{lemma}\label{boundpd}
Let $K$ be a compact subset of $U \times T^2$. Then
\[ \sup_{\cpone \times K} \left\| D^\beta \Theta^{(\nu)} \right\| < C_{\beta,K} \]
for a constant $C_{\beta,K}$ independent of $\nu$.
\end{lemma}
\begin{proof}
Lemma \ref{nubst} is the case $|\beta| := \sum \beta_i = 0$, with $\epsilon$ as $C_{0,K}$. To simplify notation, we'll drop the $K$ subindex in these constants. Assume by induction we already did this for $|\beta| = k - 1$ derivatives and for the first $\nu \geq 0$ iterations, the case $\nu = 0$ being trivial. Take partial derivatives with respect to $\theta_s$, for $s = 1, 2$ in \eqref{logthet}. This introduces a factor of the form
\begin{equation}\label{facthet}
i\frac{\partial}{\partial \theta_s} \Theta^{(\nu)}_{\gamma'}
\end{equation}
%By \eqref{bounds}, \eqref{condex} and since no $\gamma'$ appearing in the integrals for $\Theta_\gamma$ is a multiple of $\gamma$, the above can
By induction on $\nu$, the above can be bounded by $\left\| \gamma' \right\| C_{\beta'}$, where $\beta' = (1,0,0,0)$ or $(0,1,0,0)$, depending on the index $s$. When we take the partial derivatives with respect to $a$ in \eqref{logthet}, we add a factor of
\begin{equation}\label{faca}
\frac{\pi R}{\zeta'} \frac{\partial}{\partial a} Z_{\gamma'}(a) + i \frac{\partial}{\partial a} \Theta^{(\nu)}_{\gamma'}
\end{equation}
in the integrals \eqref{logthet}. Similarly, a partial derivative with respect to $\overline{a}$ adds a factor of
\begin{equation}\label{facbara}
 \pi R\zeta' \frac{\partial}{\partial \overline{a}} \overline{Z_{\gamma'}(a)} + i \frac{\partial}{\partial \overline{a}} \Theta^{(\nu)}_{\gamma'}
\end{equation}
As in \eqref{facthet}, the second term in \eqref{faca} and \eqref{facbara} can be bounded by $ \left\| \gamma' \right\| C_{\beta'}$ for $|\beta'| = 1$. Since $Z_{\gamma'}$ is holomorphic on $U \subset \cone$, and since $K \subset U \times T^2$ is compact,
\begin{equation}\label{derzeta}
\left| \frac{\partial^k}{\partial {a^k}} Z_{\gamma'} \right| \leq k! \left\| \gamma' \right\|C
\end{equation}
for all $k$ and some constant $C$, independent of $k$ and $a$. Likewise for $\overline{a}$, $\overline{Z_\gamma'}$. Thus if we take $D^{\beta} \Theta^{(\nu+1)}_\gamma$ in \eqref{logthet} for a multi-index $\beta$ with $|\beta| = k$, the right side of \eqref{logthet} becomes:
\begin{align}
-\frac{1}{4\pi }\left\{ \sum_{\gamma' > 0} f^{\gamma'} \int_r K(\zeta,\zeta') \mathcal{X}^{\text{sf}}_{\gamma'}(a,\zeta',\Theta^{(\nu)})P_{\gamma'}(a,\zeta',\theta) +  \sum_{\gamma' < 0} f^{\gamma'} \int_{-r} K(\zeta,\zeta') \mathcal{X}^{\text{sf}}_{\gamma'}(a,\zeta',\Theta^{(\nu)})Q_{\gamma'}(a,\zeta',\theta) \right\}, \label{allpartia}
\end{align}
where each $P_{\gamma'}$ or $Q_{\gamma'}$ is a polynomial obtained as follows: 

Each $\mathcal{X}^{\text{sf}}_{\gamma'}(a,\zeta',\Theta^{(\nu)})$ is a function of the type $e^g$, for some $g(a,\bar{a},\theta_1, \theta_{2})$. If $\{x_1, \ldots, x_k\}$ denotes a choice of $k$ of the variables $a,\bar{a}, \theta_1, \theta_{2}$ (possibly with multiplicities), then by the Fa\`{a} di Bruno Formula:
\begin{equation}\label{dibruno}
 \frac{\partial^k}{\partial x_1 \cdots \partial x_k} e^g = e^g \sum_{\pi \in \Pi} \prod_{B \in \pi} \frac{\partial^{|B|}g}{\prod_{j \in B}\partial x_j}:= e^g P_{\gamma'}
\end{equation}
where
\begin{itemize}
	\item $\pi$ runs through the set $\Pi$ of all partitions of the set $\{1, \ldots, k\}$.
	\item $B \in \pi$ means the variable $B$ runs through the list of all of the ``blocks'' of the partition $\pi$, and
  \item $|B|$ is the size of the block $B$.

\end{itemize}

 The resulting monomials in $P_{\gamma'}$ (same thing holds for $Q_{\gamma'}$) are products of the variables given by \eqref{facthet}, \eqref{faca}, \eqref{facbara} or their subsequent partial derivatives in $\theta, a, \overline{a}$. For each monomial, the sum of powers and total derivatives of terms must add up to $k$ by \eqref{dibruno}. For instance, when computing
 \[ \frac{\partial^3}{\partial \theta_1 \partial a^2} \mathcal{X}^{\text{sf}}_{\gamma'}(a,\zeta',\Theta^{(\nu)}) = \frac{\partial^3}{\partial \theta_1 \partial a^2} e^g, \]
 a monomial that appears in the expansion is:
 \[ \frac{\partial g}{\partial \theta_1} \left[ \frac{\partial g}{\partial a}\right]^2 = i\frac{\partial}{\partial \theta_1} \Theta^{(\nu)}_{\gamma'} \left[ \frac{\pi R}{\zeta'} \frac{\partial}{\partial a} Z_{\gamma'}(a) + i \frac{\partial}{\partial a} \Theta^{(\nu)}_{\gamma'}\right]^2 \]
 There are a total of (possibly repeated) $B_k$ monomials in $P_{\gamma'}$, where $B_k$ is the Bell number, the total number of partitions of the set $\{1, \ldots, k\}$ and $B_k \leq k!$. We can assume, without loss of generality, that any constant $C_{\beta}$ is considerably larger than any of the $C_{\beta'}$ with $|\beta'| < |\beta|$, by a factor that will be made explicit. First notice that since there is only one partition of $\{1, \ldots, k\}$ consisting of 1 block, the Fa\`{a} di Bruno Formula \eqref{dibruno} shows that $P_{\gamma'}$ contains only one monomial with the factor $D^\beta \Theta^{(\nu)}$. The other monomials have factors $D^{\beta'} \Theta^{(\nu)}$ for $|\beta'| < |\beta|$. We can do a saddle point analysis for each integrand of the form
 \[ \int_r K(\zeta,\zeta') \mathcal{X}^{\text{sf}}_{\gamma'}(a,\zeta',\Theta^{(\nu)})P^i_{\gamma'}(a,\zeta',\theta), \]
 for $P^i_{\gamma'}$ (or $Q^i_{\gamma'}$) one of the monomials of $P_{\gamma'}$ ($Q_{\gamma'}$). The saddle point analysis and the induction step for the previous $\Theta^{(\nu)}$ give the estimate
 \[ C_{\beta} \cdot \text{const} \sum_{\gamma'} \left| \left\langle \gamma, f^{\gamma'} \right\rangle \right| e^{(-2\pi R + \delta)|Z_{\gamma'}|} \]
 for the only monomial with $D^\beta \Theta^{(\nu)}$ on it. The estimates for the other monomials contain the same exponential decay term, along with powers $s$ of $C_{\beta'}, C$ such that $s \cdot |\beta'| \leq |\beta|$, and constant terms. By making $C_{\beta}$ significantly bigger than the previous $C_{\beta'}$, we can bound the entire \eqref{allpartia} by $C_{\beta}$, completing the induction step

\end{proof}

\begin{example}
To see better the estimates we obtained in the previous proof, let's consider the particular case $k = |\beta| = 3$.  If $k = 3$, there are a total of $\binom{4+3-1}{3} = 20$ different third partial derivatives for each $\Theta^{(\nu + 1)}$. There are a total of 5 different partitions of the set $\{1, 2, 3\}$ and correspondingly
\begin{align*}
\frac{\partial^3 }{\partial x_1 \partial x_2 \partial x_3 }e^g & = \\
& e^g \left[ \frac{\partial^3 }{\partial x_1 \partial x_2 \partial x_3}g + \left(  \frac{\partial^2 }{\partial x_1 \partial x_2 }g\right) \left( \frac{\partial}{\partial x_3}g\right) + \left(  \frac{\partial^2 }{\partial x_1 \partial x_3 }g\right) \left( \frac{\partial}{\partial x_2}g\right) \right. \\
& \left. + \left(  \frac{\partial^2 }{\partial x_2 \partial x_3 }g\right) \left( \frac{\partial}{\partial x_1}g\right) + \left( \frac{\partial}{\partial x_1} g \right) \left( \frac{\partial}{\partial x_2} g \right) \left( \frac{\partial}{\partial x_3} g \right) \right]
\end{align*}

If $x_1 = x_2 = x_3 = a$, 
\begin{align*}
\frac{\partial^3}{\partial a^3} \mathcal{X}^{\text{sf}}_{\gamma'}(a,\zeta',\Theta^{(\nu)}) & = \mathcal{X}^{\text{sf}}_{\gamma'}(a,\zeta',\Theta^{(\nu)}) \left[ \frac{\pi R}{\zeta'} \frac{\partial^3}{\partial a^3} Z_{\gamma'} + i \frac{\partial^3}{\partial a^3} \Theta^{(\nu)}_{\gamma'} \right.\\
& + 3 \left(\frac{\pi R}{\zeta'} \frac{\partial^2}{\partial a^2} Z_{\gamma'} + i \frac{\partial^2}{\partial a^2} \Theta^{(\nu)}_{\gamma'} \right)\left( \frac{\pi R}{\zeta'} \frac{\partial}{\partial a} Z_{\gamma'} + i \frac{\partial}{\partial a} \Theta^{(\nu)}_{\gamma'}\right) \\
& \left. + \left( \frac{\pi R}{\zeta'} \frac{\partial}{\partial a} Z_{\gamma'} + i \frac{\partial}{\partial a} \Theta^{(\nu)}_{\gamma'}\right)^3 \right]\\
& = \mathcal{X}^{\text{sf}}_{\gamma'}(a,\zeta',\Theta^{(\nu)})P(\Theta^{(\nu)}_{\gamma'})
\end{align*}

There is one and only one term containing $\frac{\partial^3}{\partial a^3} \Theta^{(\nu)}_{\gamma'}$. By induction on $\nu$, $|\frac{\partial^3}{\partial a^3} \Theta^{(\nu)}_{\gamma'}| < \left\| \gamma' \right\| C_\beta$. For the estimates of 
\[ i f^{\gamma'} \int_r K(\zeta,\zeta') \mathcal{X}^{\text{sf}}_{\gamma'}(a,\zeta',\Theta^{(\nu)})\frac{\partial^3}{\partial a^3} \Theta^{(\nu)}_{\gamma'}, \]
we do exactly the same as in the proof of Lemma \ref{nubst}. Namely, move the ray $r$ to the corresponding BPS ray $\ell_{\gamma'}$, possibly at the expense of gaining a residue bounded by
\begin{equation}\label{cbeta}
C_{\beta} \cdot \text{const} \left|  f^{\gamma'} \right| e^{(-2\pi R + \delta)|Z_{\gamma'}|}
\end{equation}
The sum of all these residues over those $\gamma'$ such that $\left\langle \gamma, \gamma'\right\rangle \neq 0$ is just a fraction of $C_{\beta}$. After moving the contour we estimate
\[ i f^{\gamma'} \int_{\ell_{\gamma'}} K(\zeta,\zeta') \mathcal{X}^{\text{sf}}_{\gamma'}(a,\zeta',\Theta^{(\nu)})\frac{\partial^3}{\partial a^3} \Theta^{(\nu)}_{\gamma'} \]
As in \eqref{spest}, we run a saddle point analysis and obtain a similar estimate \eqref{cbeta} as in Lemma \ref{nubst}. The result is that the estimate for this monomial is an arbitrarily small fraction of $C_\beta$.

If we take other monomials, like say
\[ P^1_{\gamma'} = 3 \left( \frac{\pi R}{\zeta'} \right)^2 \frac{\partial^2}{\partial a^2} Z_{\gamma'} \frac{\partial}{\partial a} Z_{\gamma'} \]
and estimate
\[ 3 f^{\gamma'}\frac{\partial^2}{\partial a^2} Z_{\gamma'} \frac{\partial}{\partial a} Z_{\gamma'} \int_r \left( \frac{\pi R}{\zeta'} \right)^2 K(\zeta,\zeta') \mathcal{X}^{\text{sf}}_{\gamma'}(a,\zeta',\Theta^{(\nu)}), \]
we do as before, computing residues and doing saddle point analysis. The difference with these terms is that partial derivatives of $Z_{\gamma'}$ are bounded by \eqref{derzeta}, and at most second derivatives of $\Theta^{(\nu)}$ (for this specific monomial, there are no such terms) appear. The extra powers of $\frac{\pi R}{\zeta'}$ that appear here don't affect the estimates, since $\mathcal{X}^{\text{sf}}_{\gamma'}$ has exponential decay on $\frac{\pi R}{\zeta'}$. The end result is an estimate of the type
\begin{equation}\label{cbetas}
C_{\beta'_1}^{s_1} \cdots C_{\beta'_m}^{s_m} C^j \cdot \text{const} \left| f^{\gamma'} \right| e^{(-2\pi R + \delta)|Z_{\gamma'}|}
\end{equation}
with all $s_i \cdot |\beta'_i|$ and $j$ $\leq |\beta|$. By induction on $|\beta|$, we can make $C_\beta$ big enough so that \eqref{cbetas} are just a small fraction of $C_\beta$. This completes the illustration of the previous proof for $\beta = (0,0,3,0)$ of the fact that $\sup |D^{\beta} \Theta^{(\nu + 1)}| < C_\beta$ on the compact set $K$.
\end{example}

Now we're ready to prove the main part of Theorem \ref{solut}, that of the existence of solutions to the Riemann-Hilbert problem.
\begin{theorem}\label{contr}
The sequence $\{\Theta^{(\nu)}\}$ converges in $\mathrsfs{X}$. Moreover, its limit $\Theta$ is piecewise holomorphic on $\zeta$ with jumps along the rays $r, -r$ and continuous on the closed half-planes determined by these rays. $\Theta$ is $C^\infty$ on $a, \overline{a}, \theta_1,  \theta_{2}$.
\end{theorem}
\begin{proof}
We first show the contraction of the $\Theta^{(\nu)}$ in the Banach space $\mathrsfs{X}$ thus proving convergence. We will use the fact that $e^x$ is locally Lipschitz and the $\Theta^{(\nu)}$ are arbitrarily close to $\theta$ if $R$ is big. In particular,
\begin{equation*}
\sup_{\zeta,\theta,a} \left| e^{i\Theta_{\gamma}^{(\nu)}} - e^{i\Theta_{\gamma}^{(\nu-1)}}\right| < \text{const} \cdot \sup_{\zeta,\theta,a} \left| \Theta_{\gamma}^{(\nu)} - \Theta_{\gamma}^{(\nu-1)}\right| \leq  \text{const} \left\| \Theta^{(\nu)} - \Theta^{(\nu-1)}\right\|,
\end{equation*}
for $\gamma$ one of the basis elements $\gamma_1,  \gamma_{2}$. For arbitrary $\gamma'$, recall that if $\gamma' =
 c_1 \gamma_1 + c_2 \gamma_{2}$, then $\Theta_{\gamma'}^{(\nu)} = c_1 \Theta_{\gamma_1}^{(\nu)} +  
c_{2} \Theta_{\gamma_{2}}^{(\nu)}$. It follows from the last inequality that
\begin{equation}\label{lips}
\sup_{\zeta,\theta} \left| e^{i\Theta_{\gamma'}^{(\nu)}} - e^{i\Theta_{\gamma'}^{(\nu-1)}} \right| < \text{const}^{\left\| \gamma' \right\|} \left\| \Theta^{(\nu)} - \Theta^{(\nu-1)}\right\|
\end{equation}

\noindent We estimate
\begin{align*}
\left\| \Theta^{(\nu+1)} - \Theta^{(\nu)}\right\|  = & \frac{1}{4\pi} \left\| \sum_{\gamma'>0} f^{\gamma'}\int_r K(\zeta,\zeta') \left[ \mathcal{X}^{\text{sf}}_{\gamma'}(a,\zeta',\Theta^{(\nu)}) -
 \mathcal{X}^{\text{sf}}_{\gamma'}(a,\zeta',\Theta^{(\nu-1)}) \right]\right.\\
 & \left. + \sum_{\gamma'<0} f^{\gamma'}\int_{-r} K(\zeta,\zeta') \left[ \mathcal{X}^{\text{sf}}_{\gamma'}(a,\zeta',\Theta^{(\nu)})-\mathcal{X}^{\text{sf}}_{\gamma'}(a,\zeta',\Theta^{(\nu-1)}  \right] \right\| \\
 & \leq  \frac{1}{4\pi } \left\| \sum_{\gamma'>0} f^{\gamma'}\int_r K(\zeta,\zeta') \left| \mathcal{X}^{\text{sf}}_{\gamma'}(a,\zeta',\theta)\right|\left| e^{i\Theta_{\gamma'}^{(\nu)}} -  e^{i\Theta_{\gamma'}^{(\nu-1)}}\right| \right\|\\
  & + \frac{1}{4\pi } \left\| \sum_{\gamma'<0} f^{\gamma'}\int_r K(\zeta,\zeta') \left| \mathcal{X}^{\text{sf}}_{\gamma'}(a,\zeta',\theta)\right|  \left| e^{i\Theta_{\gamma'}^{(\nu)}} -  e^{i\Theta_{\gamma'}^{(\nu-1)}}\right| \right\| 
\end{align*}
As in the proof of Lemma \ref{nubst}, we can move the integrals to the rays $\ell_{\gamma'}$ introducing an arbitrary small contribution from the residues. The differences of the form
\[ \left| e^{i\Theta_{\gamma'}^{(\nu)}} -  e^{i\Theta_{\gamma'}^{(\nu-1)}}\right| \]
can be expressed in terms of $\left\| \Theta^{(\nu)} - \Theta^{(\nu-1)}\right\|$ by \eqref{lips}.

The sum of the resulting integrals can be made arbitrarily small if $R$ is big by a saddle point analysis as from \eqref{stmeth} onwards. By \eqref{lips}:
\begin{align*}
\left\| \Theta^{(\nu+1)} - \Theta^{(\nu)}\right\|  & < \text{const} \left\| \sum_{\gamma'} f^{\gamma'} e^{(-2\pi R + \delta)|Z_{\gamma'}|}\right\| \left\| \Theta^{(\nu)} - \Theta^{(\nu-1)}\right\|,
\end{align*}
By making $R$ big, we get the desired contraction in $\mathrsfs{X}$ and the convergence is proved.

The holomorphic properties of $\Theta$ on $\zeta$ are clear since $\Theta$ solves the integral equation \eqref{ineqs} and the right side of it is piecewise holomorphic, regardless of the integrand.

Finally, by Lemma \ref{boundpd}, $\{D^\beta \Theta^{(\nu)}\}$ is an equicontinuous and uniformly bounded family on compact sets $K$ for any differential operator $D^\beta$ as in \eqref{dbeta}. By Arzela-Ascoli, a subsequence converges uniformly and hence its limit is of type $C^{k}$ for any $k$. Since we just showed that $\Theta^{(\nu)}$ converges, this has to be the limit of any subsequence. Thus such limit $\Theta$ must be of type $C^\infty$ on $U \times T^2$, as claimed.
\end{proof}

By Theorem \ref{contr}, the functions $\mathcal{Y}_k(a, \zeta, \theta) \ := \mathcal{X}^\text{sf}_k(a, \zeta, \Theta), k = 1, 2$ satisfy \eqref{newjumps} and condition \eqref{smoothprop}. It remains to show that the functions also satisfy the reality conditions.

\begin{lemma}\label{realcond}
For $\mathcal{Y}_k(a, \zeta, \theta)$ defined as above and with $\gamma = c_1 \gamma_1 + \gamma_2 \in \Gamma$, we define $\mathcal{Y}_{\gamma} = \mathcal{Y}_1^{c_1} \mathcal{Y}_2^{c_2}$. Then
\[ \overline{\mathcal{Y}_{-\gamma}(-1/\overline{\zeta})} = \mathcal{Y}_\gamma(\zeta) \]
\end{lemma}
\begin{proof}
Ignoring the parameters $a, \theta_1, \theta_2$ for the moment, it suffices to show 
\begin{equation}\label{thetreal}
\overline{\Theta_{k}(-1/\overline{\zeta})} = \Theta_k(\zeta), \qquad k = 1, 2
\end{equation}
We show that this is true for all $\Theta^{(\nu)}$ defined as in \eqref{logthet} by induction on $\nu$. For $\nu = 0$, $\Theta^{(0)} = (\theta_1, \theta_2)$ which are real torus coordinates and independent of $\zeta$, so \eqref{thetreal} is true. 

Assuming \eqref{thetreal} is true for $\nu$, we obtain $\Theta^{(\nu+1)}$ as in \eqref{logthet}. If we write $\zeta$ as $te^{i\varphi}, t > 0$ for some angle $\varphi$, and if we parametrize the admissible ray $r$ as $se^{i\rho}, s > 0$, then \eqref{thetreal} for $\nu+1$ follows by induction and by rewriting the integrals in \eqref{logthet} after the reparametrization $s \to \frac{1}{s}$. An essential part of the proof is the form of the symmetric kernel
\[ K(\zeta, \zeta') = \frac{d\zeta'}{\zeta'} \frac{\zeta' + \zeta}{\zeta' - \zeta} \]
which inverts the roles of $0$ and $\infty$ after the reparametrization.
\end{proof}

To verify the last property of $\mathcal{Y}_k$, we prove

\begin{lemma} 
For $\mathcal{Y}_k(a, \zeta, \theta)$ defined as above
\[ \lim_{\zeta \to 0} \mathcal{Y}_\gamma(\zeta) / \mathcal{X}^{\text{sf}}_\gamma(\zeta) \]
 exists and is real.
\end{lemma}
\begin{proof}
Write $\Theta_k^0$ for $\lim_{\zeta \to 0} \Theta_k$. In a similar way we can define $\Theta_k^\infty$. It suffices to show that $\Theta_k^0 - \theta_k$ is imaginary. This follows from Lemma \ref{realcond} by letting $\zeta \to 0$.
\end{proof}
Observe that this and the reality condition give
\[ \Theta_k^0 = \overline{\Theta_k^\infty} \]

To finish the proof of Theorem \ref{solut}, we apply the classical arguments: given two solutions $\mathcal{Y}_k, \mathcal{Z}_k$ satisfying the conditions of the theorem, the functions $\mathcal{Y}_k \mathcal{Z}^{-1}_k$ are entire functions bounded at $\infty$, so this must be a constant. By the reality condition \ref{asympt}, this constant must be real. This finishes the proof of Theorem \ref{solut}.

\section{Special Cases}\label{special}

In our choice of admissible rays $r, -r$, observe that due to the exponential decay of $\mathcal{X}^\text{sf}_k, k = 1, 2$ along these rays (see \eqref{semiflat}) and the rays $\ell_\gamma$, the jumps $S_\ell$ or $S_r, S_{-r}$ are asymptotic to the identity transformation as $\zeta \to 0$ or $\zeta \to \infty$ along these rays. Thus, one can define a Riemann-Hilbert problem whose contour is a single line composed of the rays $r, -r$, the latter with orientation opposite to the one in the previous section. The jump $S$ along the contour decomposes as $S_r, S_{-r}^{-1}$ in the respective rays and we can proceed as in the previous section with a combined contour.

\subsection{Jump Discontinuities}\label{jumpdiscnt}

In \cite{theory}, we will be dealing with a modification of the Riemann-Hilbert problem solved in \S \ref{solutions}. In particular, that paper deals with the new condition
\begin{enumerate}[label=\textnormal{(\arabic*)}]
\item[(5')]  $Z_{\gamma_2}(0) \neq 0$ for any $a$ in $U$ but $Z_{\gamma_1}$ attains its unique zero at $a = 0$.\label{zerocharge}
\end{enumerate}

Because of this condition, the jumps loose the exponential decay along those rays and they are no longer asymptotic to identity transformations. In fact, in \cite{theory} we show that this causes the jump function $S(\zeta)$ to develop a discontinuity of the first kind along $\zeta = 0$ and $\zeta = \infty$.

In this paper we obtain the necessary theory of scalar boundary-value problems to obtain solutions to this special case of Riemann-Hilbert problems appearing in \cite{theory}. We consider a general scalar boundary value problem consisting in finding a sectionally analytic function $X(\zeta)$ with discontinuities at an oriented line $\ell$ passing through 0. If $X^+(t)$ (resp. $X^-(t)$) denotes the limit from the left-hand (resp. right-hand) side of $\ell$, for $t \in \ell$, they must satisfy the boundary condition
\begin{equation}\label{boundary}
X^+(t) = G(t) X^{-}(t), \qquad t \in \ell
\end{equation}
for a function $G(t)$ that is H\"{o}lder continuous on $\ell$ except for jump discontinuities at 0 and $\infty$. We require a symmetric condition on these singularities: if $\Delta_i$, $i = 0$ or $\infty$ represents the jump of the function $G$ near any of these points,
\[ \Delta_0 = \lim_{t \to 0^+} G(\zeta) - \lim_{t \to 0^-} G(\zeta), \qquad \text{etc.} \]
Then we assume
\begin{equation}\label{symjump}
\Delta_0 = -\Delta_\infty
\end{equation}
Near 0 or $\infty$, we require for the analytic functions $X^+(\zeta), X^-(\zeta)$ to have only one integrable singularity of the form
\begin{equation}\label{estm}
|X^{\pm}(\xi)| < \frac{C}{|\xi|^\eta}, \qquad (0 \leq \eta < 1) 
\end{equation}
For $\xi$ a coordinate of $\cpone$ centered at either 0 or $\infty$. By \eqref{estm}, each function $X^\pm$ is asymptotic to 0 near the other point in the set $\{0, \infty\}$.

\begin{lemma}\label{jumpdisc}
There exists functions $X^+(\zeta), X^-(\zeta)$, analytic on opposite half-planes on $\cone$ determined by the contour $\ell$ and continuous on the closed half-planes such that, along $\ell$, the functions obey \eqref{boundary} and \eqref{estm}, with a H\"{o}lder continuous jump function $G(t)$ satisfying \eqref{symjump}. The functions $X^+(\zeta), X^-(\zeta)$ are unique up to multiplication by a constant.
\end{lemma}
\begin{proof}
We follow \cite{gakhov} for the solution of this exceptional case. As seen above, we only have jump discontinuities at $0$ and $\infty$. For any point $t_0$ in the contour $\ell$, and a function $f$ with discontinuities of the first kind on $\ell$ at $t_0$, we denote by $f(t_0 - 0)$ (resp. $f(t_0+0)$) the left (resp. right) limit of $f$ at $t_0$, according to the given orientation of $\ell$.

Let
\[ \eta_0 = \frac{1}{2\pi i} \log \frac{G(0-0)}{G(0+0)} \]
Similarly, define
\[ \eta_\infty = \frac{1}{2\pi i} \log \frac{G(\infty-0)}{G(\infty+0)} \]
 Since $G$ obeys condition \ref{symjump}, $\eta_0 = - \eta_\infty$. Observe that by definition, $|\eta_0| < 1$, and hence the same is true for $\eta_\infty$.
 
Let $D^+$ be the region in $\cpone$ bounded by $\ell$ with the positive, counterclockwise orientation. Denote by $D^-$ the region where $\ell$ as a boundary has the negative orientation. We look for solutions of the homogeneous boundary problem \eqref{boundary} . To solve this, pick a point $\zeta_0 \in D^+$ and introduce two analytic functions 
\[ (\zeta - \zeta_0)^{\eta_0}, \qquad \zeta^{\eta_0} \]
Make a cut in the $\zeta$-plane from the point $\zeta_0$ to $\infty$ through 0, with the segment of the cut from $\zeta_0$ to 0 wholly in $D^+$. Consider the functions 
\begin{equation*}
\omega^+(\zeta) = \zeta^{\eta_0}, \qquad \omega^- = \left( \frac{\zeta}{\zeta - \zeta_0} \right)^{\eta_0}
\end{equation*}
Due to our choice of cut, $\omega^+$ is analytic in $D^+$ and $\omega^-$ is analytic in $D^-$. Introduce new unknown functions $Y^\pm$ setting
\begin{equation}\label{eta0}
X^\pm (\zeta) = \omega^\pm (\zeta) Y^\pm (\zeta)
\end{equation}
The boundary condition \eqref{boundary} now takes the form
\begin{equation}\label{contb}
Y^+(t) = G_1(t) Y^-(t), \qquad t \in \ell
\end{equation}
where
\begin{equation*}
G_1(t) = \frac{\omega^-(t)}{\omega^+(t)} G(t) = (t - \zeta_0)^{-\eta_0} G(t), \qquad t \in \ell
\end{equation*}
By the monodromy of the function $(\zeta - \zeta_0)^{-\eta_0}$ around 0 and infinity and since $\eta_\infty = -\eta_0$, it follows that $G_1$ is continuous in the entire line $\ell$. Hence, we reduced the problem \eqref{boundary} to a problem \eqref{contb} with continuous coefficient, which can be solved with classical Cauchy integral methods.

By assumption, we seek solutions of \eqref{boundary} with only one integrable singularity i.e. estimates of the form \eqref{estm}. The notion of \textit{index} (winding number) for $G(t)$ in the contour $\ell$ is given by (see \cite{gakhov}) $\varkappa = \left \lfloor{\eta_0}\right \rfloor +  \left \lfloor{\eta_\infty}\right \rfloor  + 1 = 0$, so the usual method of solution of \eqref{contb} as
\[ Y = \exp \left( \frac{1}{2\pi i} \int_\ell K(\zeta',\zeta) \log G_1(\zeta') \right) \]
(for a suitable kernel $K(\zeta',\zeta)$ that makes the integral along $\ell$ convergent) needs no modification. We can also see from \eqref{eta0} that $X^\pm$ has an integrable singularity at 0 (resp. $\infty$) if $\eta_0$ is negative (resp. positive).
 
 We need to show that for different choices of $\zeta_0 \in D^+$ the solutions $X$ only differ by a constant. To see this, by taking logarithms in \eqref{boundary} it suffices to show uniqueness of solutions to the homogeneous additive boundary problem
 \begin{equation}\label{homg}
 \Phi^+(t) - \Phi^-(t) = 0, \qquad t \in \ell
 \end{equation}
 and with the assumption that $\Phi$ vanishes at a point and at the points of discontinuity of $G(\zeta)$, $\Phi^\pm$ satisfies an estimate as in \eqref{estm}. The relation \eqref{homg} indicate that the functions $\Phi^+, \Phi^-$ are analytically extendable through the contour $\ell$ and, consequently, constitute a unified analytic function in the whole plane. This function has, at worst, isolated singularities but according to the estimates \eqref{estm}, these singularities cannot be poles or essential singularities, and hence they can only be branch points. But a single valued function with branch points must have lines of discontinuity, which contradicts the fact that $\Phi^+ = \Phi^-$ is analytic (hence continuous) on the entire plane except possibly at isolated points. Therefore, the problem \eqref{homg} has only the trivial solution.
\end{proof}

\subsection{Zeroes of the boundary function}

Because of condition \ref{zerocharge}, yet another special kind of Riemann-Hilbert problem arises in \cite{theory}. We still want to find a
sectionally analytic function $X(\zeta)$ satisfying the conditions \eqref{boundary} with $G(t)$ having jump discontinuities at $0, \infty$ with
the properties \eqref{symjump} and \eqref{estm}. In this subsection, we allow the case of $G(t)$ having zeroes of integer order on finitely many
points $\alpha_1, \ldots, \alpha_\mu$ along $\ell$. Thus, we consider a Riemann-Hilbert problem of the form
\begin{equation}\label{zerog}
X^+(t) = \prod_{j = 1}^\mu (t - \alpha_j)^{m_j} G_1(t) X^-(t), \qquad t \in \ell
\end{equation}
where $m_j$ are integers and $G_1(t)$ is a non-vanishing function as in \S \ref{jumpdiscnt}, still with discontinuities at 0 and $\infty$ as in \eqref{symjump}. 

\begin{lemma}
For a scalar Riemann-Hilbert problem as in \ref{zerog} and with $G_1(t)$ a non-vanishing function with discontinuities of the first kind at $0$ and $\infty$ obeying \eqref{symjump}, there exist solutions $X^{\pm}(\zeta)$ unique up to multiplication by a constant. At all points $\alpha_j$ as above, both analytic functions $X^+(\zeta), X^-(\zeta)$ are bounded and $X^+$ has a zero of order $m_j$.
\end{lemma}
\begin{proof}
By Lemma \ref{jumpdisc}, there exists non-vanishing analytic functions $Y^+(\zeta), Y^-(\zeta)$ on opposite half-planes $D^+, D^-$ determined by $\ell$ and continuous along the boundary such that
\begin{equation*}
G_1(t) = \frac{Y^+(t)}{Y^-(t)}, \qquad t \in \ell
\end{equation*}
We can define 
\begin{align}
X^+(\zeta) & =  \prod_{j = 1}^\mu (\zeta - \alpha_j)^{m_j} Y^+(\zeta) \label{xplus}\\
X^-(\zeta) & =   Y^-(\zeta) \label{xminus}
\end{align}
This clearly satisfies \eqref{zerog} and, since $Y^+$ is non-vanishing on $D^+$, it shows that $X^+$ has a zero of order $m_j$ at $\alpha_j \in \ell$. To show uniqueness of solutions, note that if $X^+, X^-$ are any solutions to the Riemann-Hilbert problem, we can write the boundary condition \eqref{zerog} in the form
\begin{equation*}
\dfrac{X^+(t)}{Y^+(t) \displaystyle \prod_{j = 1}^\mu (t - \alpha_j)^{m_j}} = \frac{X^-(t)}{Y^-(t)}, \qquad t \in \ell
\end{equation*}
The last relation indicates that the functions
\[ \dfrac{X^+(\zeta)}{Y^+(\zeta) \displaystyle \prod_{j = 1}^\mu (\zeta - \alpha_j)^{m_j}}, \frac{X^-(\zeta)}{Y^-(\zeta)}  \]
are analytic in the domains $D^+, D^-$ respectively and they constitute the analytical continuation of each other through the contour $\ell$. The points $\alpha_j$ cannot be singular points of this unified analytic function, since this would contradict the assumption of boundedness of $X^+$ or $X^-$. The behavior of $X^{\pm}$ at $0$ or $\infty$ is that of $Y^\pm$, so by Liouville's Theorem,
\[ \dfrac{X^+(\zeta)}{Y^+(\zeta) \displaystyle \prod_{j = 1}^\mu (\zeta - \alpha_j)^{m_j}} = \frac{X^-(\zeta)}{Y^-(\zeta)} = C \]
for $C$ a constant. This forces $X^+(\zeta), X^-(\zeta)$ to be of the form \eqref{xplus}, \eqref{xminus}.
\end{proof}

\bibliographystyle{amsplain}
\bibliography{diss}        % is inserted.			     %

%\printbibliography
\end{document}